\documentclass[a4paper,reqno]{amsart}

\usepackage[all,2cell]{xy}
\usepackage[pagebackref,colorlinks]{hyperref}
\usepackage{graphicx}
\usepackage{color}
\usepackage{subfigure}
\usepackage{amssymb}
\usepackage{amsthm}
\usepackage{amsmath}
\usepackage{amsfonts}
\usepackage{mathrsfs}
\usepackage{hyperref}

\hypersetup{hypertex=red,colorlinks=true,linkcolor=red,anchorcolor=red,citecolor=red,CJKbookmarks=true}

\usepackage[title,titletoc]{appendix}
\usepackage{tikz}
\usetikzlibrary{backgrounds,mindmap}
\usepackage{xcolor}
\usetikzlibrary{calc,positioning,intersections}
\usepackage{pgfplots}
\usepackage{listings}

\allowdisplaybreaks[4]

\usepackage{indentfirst}

\usepackage{lineno}

\numberwithin{equation}{section}

\theoremstyle{plain}
\newtheorem{theorem}{Theorem}[section]
\newtheorem{proposition}[theorem]{Proposition}
\newtheorem{lemma}[theorem]{Lemma}
\newtheorem{corollary}[theorem]{Corollary}
\theoremstyle{remark}
\newtheorem{remark}[theorem]{\bf Remark}

\theoremstyle{definition}
\newtheorem{definition}[theorem]{Definition}
\newtheorem{example}[theorem]{Example}

\begin{document}

	\title{Graded isomorphisms of Leavitt path algebras and Leavitt inverse semigroups}
	
	\author{Huanhuan Li}
	\address{ Center for Pure Mathematics\\School of Mathematical Sciences\\Anhui University, Hefei, 230601, China}
    \email{lhh@ahu.edu.cn}
	
	\author{Zongchao Li}
	\address{School of Mathematical Sciences\\Anhui University, Hefei, 230601, China}
	\email{leezongc@163.com}

	\author{Zhengpan Wang}
	\address{School of Mathematics and Statistics\\Southwest University, Chongqing, 400715, China}
	\email{zpwang@swu.edu.cn}

	
	
	\date{\today}

	\begin{abstract}
		 Leavitt inverse semigroups of directed finite graphs are related to Leavitt graph algebras of (directed) graphs. Leavitt path algebras of graphs have the natural  $\mathbb  Z$-grading via the length of paths in graphs.  We consider the $\mathbb  Z$-grading on Leavitt inverse semigroups. For connected finite graphs having vertices out-degree at most $1$, we give a combinatorial sufficient and necessary condition on graphs to classify the corresponding Leavitt path algebras and  Leavitt inverse semigroups up to graded isomorphisms. More precisely, the combinatorial condition on two graphs coincides if and only if the Leavitt path algebras of the two graphs are $\mathbb  Z$-graded isomorphic if and only if the Leavitt inverse semigroups of the two graphs are $\mathbb  Z$-graded isomorphic.
	\end{abstract}
	
	\maketitle
	\section{Introduction}
	Leavitt path algebras were introduced  by Ara, Moreno and Pardo \cite{r5}, and Abrams and Aranda Pino \cite{r6} independently. Leavitt path algebras are an outgrowth of a class of algebras defined by Leavitt \cite{r8}, called Leavitt algebras. A Leavitt path algebra associating to a directed graph is a $\mathbb  Z$-graded algebra. Hazrat \cite{r9} studied this $\mathbb  Z$-grading and characterized the $\mathbb  Z$-graded algebraic structure of Leavitt path algebras associated to polycephaly graphs via graded matrix rings. Hazrat and Mesyan \cite{r10} systematically developed the theory of graded semigroups, that is, semigroups $S$ partitioned by a group $G$, in a manner compatible with the multiplication on $S$.
	
	The authors of \cite{r1} introduced a class of inverse semigroups for directed graph that they refer to as Leavitt inverse semigroups.  They proved that for two connected graphs $E$ and $F$ whose vertices have out-degree at most 1, the Leavitt inverse semigroups of $E$  and $F$ are isomorphic if and only if the Leavitt path algebras of $E$ and $F$ are isomorphic as algebras if and only if $E$ and $F$ have the same number of vertices; \cite[Theorem 4.7]{r1}. Leavitt path algebras have the natural $\mathbb Z$-grading which is given by  the length of paths  in graphs. We explore the natural $\mathbb Z$-grading on the Leavitt inverse semigroups of graphs.
    Under the same condition as \cite[Theorem 4.7]{r1}, we prove that  the Leavitt inverse semigroups of $E$  and $F$ are  $\mathbb Z$-graded  isomorphic if and only if the Leavitt path algebras are  $\mathbb Z$-graded isomorphic.  We also present a combinatorial sufficient and necessary condition on graphs to classify the corresponding Leavitt path algebras and  Leavitt inverse semigroups up to graded isomorphisms.
	
 	
	This note is structured as follows. In  section \ref{section1}, we recall definitions of Leavitt path algebras and Leavitt inverse semigroups of finite graphs. We give $\mathbb Z$-grading on Leavitt inverse semigroups. In section \ref{section3}, we give a combinatorial sufficient and necessary condition to classify Leavitt path algebras and  Leavitt inverse semigroups of finite connected graphs whose vertices have out-degree at most $1$ up to graded isomorphisms; see Theorem \ref{theorem}.

	\section{Graded inverse semigroups and Leavitt path algebras}
\label{section1}

In this section, we recall definitions of Leavitt path algebras and Leavitt inverse semigroups of finite graphs. We also consider $\mathbb Z$-grading on Leavitt inverse semigroups.
\subsection{Graphs}
	 
	A directed graph $E = (E^{0}, E^{1}, s, r)$ consists of two sets $E^{0}, E^{1}$ and two maps $s ,r : E^1 \to E^{0}$. We denote the set of vertices by $E^{0}$ and the set of edges by $E^{1}$. The maps $s$, $r$ can be respectively called the source mapping and the range mapping. For each vertex $v\in E^{0}$, $s^{-1}(v) = \left\lbrace e \in E^{1} | s(e) = v \right\rbrace $ and the out-degree of a vertex $v$ is $|s^{-1}(v)|$. A vertex $v\in E^{0}$ with out-degree 0 is called a sink. A path $p$ in a graph $E$ is a sequence $p = e_{1}e_{2}\dots e_{n}$ of edges $e_{i}\in E^{1}$ such that $r(e_{i})=s(e_{i+1})$, $i = 1,2,\dots,n-1$. In this case, $s(p) = s(e_{1})$ is the source of $p$, $r(p) = r(e_{n})$ is the range of $p$, and $n=|p|$ is the length of $p$. We can consider the vertices of $E$ as empty paths and these empty paths have length $0$. For $n \geq 2$ we define $E^{n}$ as a set of paths in $E$ of length $n$, and define $\text{Path}(E) = \bigcup_{n\geq 0}E^{n}$, the set of all paths in $E$. 
 In this note, we will only consider directed finite graphs, namely a directed graph has finitely many vertices and finitely many edges.

	
	For the convenience of study, the notation can be extended to allow paths in which edges are read in either the positive or negative direction. To do this, we associate with each edge $e$ an inverse edge $e^{*}$. We denote the set $\left\lbrace e^{*} | e \in E^{1} \right\rbrace $ by $(E^{1})^{*}$ and define $s(e^{*})= r(e) $ and $r(e^{*})=s(e)$. With this convention, we define a walk in $E$ as a sequence $q = e_{1}e_{2}\dots e_{n}$ with edges $e_{i} \in E^{1}\bigcup (E^{1})^{*}$ such that $r(e_{i})=s(e_{i+1})$ for $i = 1,2,\dots,n-1$. Similarly we define $s(q) = s(e_{1})$ and $r(q) = r(e_{n})$. We say that graph $E$ is connected if for all $ v$, $w \in E^{0}$ there exists at least one walk $q$ with $s(q)=v$ and $r(q)=w$. A walk $p$ based at $v$ is a circuit if $s(p) = r(p) = v.$ A path $p = e_{1}e_{2}\dots e_{n}$ is said to be a cycle if $p$ is a circuit and $s(e_{i}) \neq s(e_{j})$ for every $i\neq j$.  
 Two cycles $C_{1}$ and $C_{2}$ are said to be conjugate if $C_{1} = e_{1}e_{2}\dots e_{n}$ and $C_{2} = e_{i}e_{i+1}\dots e_{n}e_{1}\dots e_{i-1}$ for some $i$. A walk $q = e_{1}e_{2}\dots e_{n}$ is called reduced if $e_{i} \neq e_{i+1}^{*}$ for each $i$. A reduced circuit is a circuit $q = e_{1}e_{2}\dots e_{n}$ that is a reduced walk and such that $e_{1} \neq e_{n}^{*}$. The graph $E$ is said to be acyclic if it has non-trivial cycles. $E$ is called a tree if it is connected and has no non-trivial reduced circuits.  
	
	\begin{proposition} \cite[Proposition 2.1]{r1}
		\label{prop1}
		Let $E$ be a connected graph whose vertices have out-degree at most 1. Then
		
		(a) $E$ has at most one sink. If $E$ does have a sink $v_{0}$ and $v$ is any other vertex in $E$, then there is a unique path from $v$ to $v_{0}$ and $E$ is a tree.
		
		(b) If $E$ is not a tree then $E$ has a non-trivial cycle and any two non-trivial cycles are cyclic conjugates of each other. Furthermore, if $v'$ is any vertex on one of these cycles $C$ and $v$ is any other vertex of $E$ then there is a unique path from $v$ to $v'$ that does not include the cycle $C$ as a subpath.
	\end{proposition}
	

	In the following we give two simple graphs whose vertices have out-degree at most $1$.
	
	\begin{example}
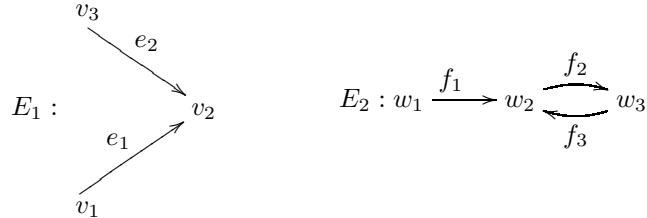
		\begin{equation*}
			{\def\labelstyle{\displaystyle}
				\xymatrix{	
					\qquad		 v_{3}\ar[dr]^{e_{2}} &\\
					E_{1}:	\qquad	&v_{2}\\
					\qquad		 v_{1}\ar[ur]^{e_{1}}&\\
			} }	\qquad \qquad
			{\def\labelstyle{\displaystyle}
				\xymatrix{
					\\
					E_{2}:	w_{1}\ar[r]^{f_{1}} &w_{2}\ar@/^/[r]^{f_{2}}&w_{3}\ar@/^/[l]^{f_{3}}	
			} }	
		\end{equation*}
	
	\end{example}

	\subsection{Leavitt inverse semigroups and Leavitt path algebras}

	Let $K$ be a field and $E$ be a finite graph. Now we recall the definition of Leavitt inverse semigroup of $E$.
	
	\begin{definition}\label{def1}  (\cite[Theorem 1.2]{r12},\cite[\S{4}]{r1})
		The Leavitt inverse semigroup of $E$, denoted by $LI(E)$, is the semigroup with zero generated by the $E^{0} \cup E^{1}\cup (E^{1})^{*}$, subject to the following relations :
  
  \noindent(1) $s(e)e=er(e)=e$ for all $e\in E^{0}\cup E^1\cup (E^1)^*$;\\
		(2) $uv=0$ if $u,v\in E^0$ and $u\neq v$;\\
		(3) $e^{*}f=0$ if $e,f\in E^1$ and $e\neq f$;\\
		(4) $e^*e=r(e)$ if $e\in E^{1}$;\\
		 (5) $v= ee^{*}$ for each $v\in E^{0}$ of out-degree 1.
	\end{definition}	

 Now we recall the forms of elements in the Leavitt inverse semigroups.

    \begin{lemma}\label{lem forms}  \cite[Theorem 4.2]{r1}
       Every element of $LI(E)$ is uniquely expressible in one of the forms

        (a) $pq^*$ where $p=e_{1}\dots e_{n}$ and $q=f_{1}\dots f_{m}$ are (possibly empty) paths with $r(e_n)=r(f_m)$ and $e_n \neq f_m$; or

        (b) $pq^*= p'ee^{*}q'^{*}$ where $p'$ and $q'$ are (possibly empty) paths with $r(p') = r(q')$ and the vertex $s(e)=r(p') = r(q')$ has out-degree at least 2.
    \end{lemma}

  We give an example of the Leavitt inverse semigroup as follows.
	\begin{example} {\label{example1}} Let $F$ be the following graph.
		\begin{equation*}
			{\def\labelstyle{\displaystyle}
				
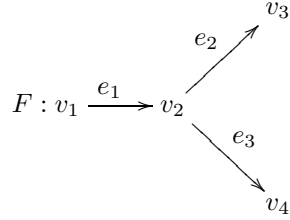
\xymatrix{	
					& & v_{3}\\
					F : v_{1}\ar[r]^{e_{1}} &v_{2}\ar[ur]^{e_{2}} \ar[dr]^{e_{3}}&\\
					& & v_{4}\\						
			} }		
		\end{equation*}All the non-zero elements of $LI(F)$ are 	
		$$v_{1},
		v_{2}, e_{1}^{*},
		v_{3}, e_{2}^{*}, e_{2}^{*}e_{1}^{*},
        v_{4}, e_{3}^{*}, e_{3}^{*}e_{1}^{*},e_{1}, $$
		$$
		e_{2}, e_{2}e_{2}^{*}, e_{2}e_{2}^{*}e_{1}^{*},
        e_{3}, e_{3}e_{3}^{*}, e_{3}e_{3}^{*}e_{1}^{*},
        e_{1}e_{2}, e_{1}e_{2}e_{2}^{*}, e_{1}e_{2}e_{2}^{*}e_{1}^{*},
        e_{1}e_{3}, e_{1}e_{3}e_{3}^{*}, 
    e_{1}e_{3}e_{3}^{*}e_{1}^{*}.$$
\end{example}

We refer to \cite[Definition 1.1.1]{r4} 
 for the definition of a Leavitt path algebra of a finite graph over a field.

\begin{definition}
  Let $E$ be a finite graph and $K$ a field. The Leavitt path algebra of $E$ over $K$, denoted by $L_K(E)$, is the free associative algebra generated by $E^{0} \cup E^{1}\cup (E^{1})^{*}$ with coefficients in $K$, subject to the relations (1)--(4) used to define the Leavitt inverse semigroup $LI(E)$ and the additional relation:
\begin{itemize}
\item[(6)] $w = \sum\limits_{\left\{e \in E^1 \,|\, s(e) = w \right\}} e e^{*} \text{~for every~} w \in E^0 \text{~which is not a sink}$;
\end{itemize}
\end{definition}

The relations (3), (4) and (6) are called Cuntz-Krieger relations. The elements in $(E^{1})^{*}$ are called \textit{ghost edges}. If $p = e_1 e_2 \cdot \cdot \cdot e_n$ is a path in $\text{Path}(E) $, we define $p^{*}$ = $e_n^{*} e_{n-1}^{*} \cdot \cdot \cdot e_1^{*}$. In particular, we define $v^{*} = v$ for all $v \in E^0$. Let $x$ be an arbitrary element in $L_K(E)$.  One observes that $x$ can be written as $\sum_{i}\lambda_i p_i q_i^{*}$, where $\lambda_i \in K$ and $p_i, q_i \in \text{Path}(E) $ and $r(p_i) = r(q_i)$.

We refer the reader to \cite{r2, NvObook} for the theory of graded rings. Let $G$ be a group with identity denoted by $0$. A ring $R$ with unit
is called a \emph{$G$-graded ring} if $ R=\bigoplus_{ \gamma \in G} R_{\gamma}$
such that each $R_{\gamma}$ is an additive subgroup of $R$ and $R_{\gamma}  R_{\delta}
\subseteq R_{\gamma + \delta}$ for all $\gamma, \delta \in G$. The group $R_\gamma$ is
called the $\gamma$-\emph{homogeneous component} of $R.$ When it is clear from context that a ring $R$ is graded by the group $\Gamma,$ we simply say that $R$ is a  \emph{graded
ring}. We denote the set of all homogeneous elements of the graded ring $R$, by $R^h$.

The Leavitt path algebra has a natural $\mathbb{Z}$-grading.
Set $\deg v=0, v\in E^0$, $\deg e=1$, and $\deg e^*=-1, e\in E^1$. Then $L_K(E)=\bigoplus_{i\in\mathbb Z} L_K(E)_i $ where
   $$ L_K(E)_i = \operatorname{span}(\left\{ p q^{*} \,|\, p, q \in \text{Path}(E), |p| - |q| = i \text{~and~}r(p) = r(q)\right\}). $$

	Recall that a basis of $L_K(E)$  as $K$-vector space was given by \cite[Theorem 1]{r16}. We write down the bases of $L_K(E)$ as $K$-vector spaces under the condition that $E$ is a connected finite graph whose vertices have out-degree at most $1$. In this case for each non-sink vertex there exists only one edge starting from it. So each edge in $E$ is special, equivalently the map $\gamma: E^0\xrightarrow[]{} E^1$ can be uniquely defined. Then the basis elements of $L_K(E)$ are given as follows: (i) $v,v\in E^0$; (ii) $p,p^*$ , where $p$ is a path in $E$ with $|p|\geq 1$; (iii) $pq^*,$ where $p=e_1e_2\cdots e_n$ and $q=f_1f_2\cdots f_m$ are paths in $E$ ending at the same vertex $r(e_n)=r(f_m)$ with the condition that the last edges $e_n$ and $f_m$ are distinct.
 
 Let $E$ be a connected finite graph whose vertices have out-degree at most $1$.  If $E$ is a tree, by Proposition \ref{prop1} there exists a unique sink $v_{0}$. Let $\left\lbrace p_{1}, p_{2},\dots,p_{n}\right\rbrace $ be the set of all path ending at $v_{0}$. It follows from \cite[Lemma 3.4]{r7} that all $p_{i}p_{j}^{*}$, $1\leq i,j\leq n$,are the basis elements for $L_K(E)$. If $E$ is not a tree, then $E$ has a only non-trivial cycle $C$. We choose $v$ (an arbitrary vertex) in $C$ and remove the edge $e$ with $s(e) = v$ from the cycle $C$. In this new graph, let $\left\lbrace p_{1}, p_{2},\dots,p_{m} \right\rbrace$ be the set of all paths ending at $v$.  It follows from \cite[Theorem 1.6.21]{r2} that all $p_{i}C^{k}p_{j}^{*}$, $1\leq i,j\leq n$, $k\in \mathbb{Z}$, are the basis elements for $L_K(E)$. 
	
	By Lemma \ref{lem forms} we obtain the following consequence immediately; also compare  \cite[Theorem 4.7]{r1}.

	\begin{lemma}
		\label{lemma 11}
		Let $E$ be a connected finite graph whose vertices have out-degree at most 1, then every non-zero element of $LI(E)$ is uniquely expressible in one of the forms

  (i) $v,v\in E^0$; 
  
  (ii) $p,p^*,$ where $p$ is a path in $E$ with $|p|\geq 1$; 
  
  (iii) $pq^*$, where $p=e_1e_2\cdots e_n$ and $q=f_1f_2\cdots f_m$ are paths in $E$ ending at the same vertex $r(e_n)=r(f_m)$ with the condition that the last edges $e_n$ and $f_m$ are distinct.

  Therefore the basis of the Leavitt path algebra $L_{K}(E)$ consists of all the non-zero elements of the Leavitt inverse semigroup $LI(E).$
	\end{lemma}

	\begin{definition} \cite[Definition 2.1]{r10}
		Let $S$ be a semigroup and $G$ a group. If there is a map $\varphi: S  \backslash \left\lbrace 0 \right\rbrace \to G $ such that $\varphi(st) = \varphi(s)\varphi(t)$, whenever $ st \ne 0 $, then we call $S$ a $G$-graded semigroup and $\varphi$ a grading map of $S$.
		For each $\alpha \in G$, we set  $$S_{\alpha}:= \varphi^{-1}(\alpha)\cup \left\lbrace 0\right\rbrace .$$
		
		Equivalently, $S$ is a $G$-graded semigroup if there exist subsets $S_{\alpha} (\alpha \in G)$ of $S$ such that  $$S=\bigcup _{\alpha \in G}S_{\alpha} $$		
		\noindent where $S_{\alpha}S_{\beta} \subseteq S_{\alpha \beta}$ for all $\alpha,\beta \in G$, and $S_{\alpha}\cap S_{\beta} = \left\lbrace 0\right\rbrace $ for all distinct $\alpha , \beta \in G $.
	\end{definition}
	Let $S$ be a $G$-graded semigroup. For each $g\in G$, we refer to $S_{g}$ as the component of $S$ of degree $g$. If $s\in S_{g}\backslash \left\lbrace 0\right\rbrace $, we say that the degree of $s$ is $g$, and write $\deg s = g$.
	We recall that a homomorphism $\varphi: S \to T$ of $G$-graded semigroups is a graded homorphism if $\varphi(S_{\alpha})\subseteq T_{\alpha}$ for every $\alpha \in G$. Thus, a graded homomorphism is a homomorphism that preserves the degrees of the elements. Futhermore we call $\varphi$ a graded isomorphism if $\varphi$ is a bijection.

		For each  graph $E$, the Leavitt inverse semigroup $LI(E)$ is a $\mathbb  Z$-graded semigroup. By Lemma \ref{lem forms},  every element of $LI(E)$ is uniquely expressible in one of the forms

        (a) $pq^*$ where $p=e_{1}\dots e_{n}$ and $q=f_{1}\dots f_{m}$ are (possibly empty) paths with $r(e_n)=r(f_m)$ and $e_n \neq f_m$; or

        (b) $pq^*= p'ee^{*}q'^{*}$ where $p'$ and $q'$ are (possibly empty) paths with $r(p') = r(q')$ and the vertex $s(e)=r(p') = r(q')$ has out-degree at least 2.
We define a map
		\begin{align*}
			\phi: LI(E)\backslash \left\lbrace 0\right\rbrace  &\longrightarrow \mathbb  Z\\
			pq^{*} &\longmapsto |p|-|q|.
		\end{align*}
	where $pq^*$ is in one of the above two forms. Next we prove that $\phi$ is a grading map. We suppose that $p_{1}q_{1}^{*}, p_{2}q_{2}^{*}$ are non-zero elements of $LI(E)$ written in the above two forms. Then the product of $p_{1}q_{1}^{*}$ and $p_{2}q_{2}^{*}$ in $LI(E)$ is computed as follows:
		\begin{equation}
  \label{fourcases}
  (p_{1}q_{1}^{*})(p_{2}q_{2}^{*}) = \begin{cases}
			p_{1}p_3 q_{2}^{*}, & \text{if} \quad p_{2}=q_{1}p_3 \text{ for some } p_3 \in \text{Path}(E) \text{~and~} |p_3|\geq 1,\\
			p_{1}q_3^{*}q_{2}^{*}, & \text{if}  \quad q_{1}= p_{2}q_3 \text{ for some } q_3 \in \text{Path}(E)\text{~and~} |q_3|\geq 1,\\
   p_1'q_2'^*,& \text{if} \quad p_{2}=q_{1},p_1=p_1'e_t\cdots e_1, q_2=q_2'e_t\cdots e_1,\\
			0 ,& \text{otherwise}.
		\end{cases} 
  \end{equation} Here in the third subcase of \eqref{fourcases}, $e_1,\cdots, e_t$ with $t\geq 0$ are the edges with out-degree 1 and $p'_1q_2'^*$ in one of above two forms are transformed from $p_1q_2^*$ using the relation (5) of $LI(E)$. When $t=0$, we have $p'_1q_2'^*=p_1q_2^*.$
		In the first subcase of \eqref{fourcases}, we have $$\phi((p_{1}q_{1}^{*})(p_{2}q_{2}^{*})) = \phi(p_{1}p_3 q_{2}^{*}) = |p_{1}|+|p_3|-|q_{2}| = |p_{1}| - |q_{1}|+ |p_{2}|-|q_{2}| = \phi(p_{1}q_{1}^{*})\phi(p_{2}q_{2}^{*}).$$ In the second subcase of \eqref{fourcases}, we have $$\phi((p_{1}q_{1}^{*})(p_{2}q_{2}^{*})) = \phi(p_{1}q_3^{*} q_{2}^{*}) = |p_{1}|-|q_3|-|q_{2}| = |p_{1}|-|q_{1}|+|p_{2}|-|q_{2}|=\phi(p_{1}q_{1}^{*})\phi(p_{2}q_{2}^{*}).$$ In the third subcase of \eqref{fourcases}, we have  
  
  $$\phi((p_{1}q_{1}^{*})(p_{2}q_{2}^{*}))=\phi(p_1'q_2'^*)=(|p_1|-t)-(|q_2|-t)=\phi(p_{1}q_{1}^{*})\phi(p_{2}q_{2}^{*}).$$ Hence $\phi$ is a $\mathbb Z$-grading map of $LI(E)$.

	\section{Graded isomorphisms of Leavitt path algebras and Leavitt inverse semigroups}
	\label{section3}
	
	In this section, we give a combinatorial sufficient and necessary condition to classify Leavitt path algebras and  Leavitt inverse semigroups of graphs whose vertices have out-degree at most $1$ up to graded isomorphisms.

		\subsection{Graded isomorphisms of Leavitt path algebras} In this subsection, we consider the graded algebraic structure of Leavitt path algebras. We write down Lemmas \ref{lemmmma 4.8} and \ref{lemmmma 4.9} which were given by Hazrat \cite{r9}. 
		
	We first recall a grading on matrix rings. Given an abelian group $G$ and a $G$-graded ring $R$. Let $\{ 
 \gamma_1, \dots, \gamma_n \}$ be a subset of $G$ and $x$ be a homogeneous element of $R$.
 Define a grading on the $n\times n$-matrix ring $M_n(R)$ by assigning $$\text{deg}(e_{ij}(x))=\text{deg}(x)+\gamma_i-\gamma_j$$ and extend it linearly such that we obtain a graded matrix ring $M_n(R)(\gamma_1, \gamma_2,\cdots, \gamma_n)$. Here $e_{ij}(x)$ is the matrix with $x$ in the $ij$-position and zero elsewhere. One can see  that for $\lambda \in G$ $$ M_n(R)(\gamma_1, \gamma_2,\cdots, \gamma_n)_{\lambda}=\begin{pmatrix}
     R_{\lambda+\gamma_{1}-\gamma_{1}}&R_{\lambda+\gamma_{2}-\gamma_{1}}&\dots&R_{\lambda+\gamma_{n}-\gamma_{1}}\\
     R_{\lambda+\gamma_{1}-\gamma_{2}}&R_{\lambda+\gamma_{2}-\gamma_{2}}&\dots&R_{\lambda+\gamma_{n}-\gamma_{2}}\\
    \vdots&\vdots&\ddots&\vdots\\
     R_{\lambda+\gamma_{1}-\gamma_{n}}&R_{\lambda+\gamma_{2}-\gamma_{n}}&\dots&R_{\lambda+\gamma_{n}-\gamma_{n}}
 \end{pmatrix} $$ and  $M_n(R)(\gamma_1, \gamma_2,\cdots, \gamma_n)= \bigoplus_{\lambda\in G}M_n(R)(\gamma_1, \gamma_2,\cdots, \gamma_n)_{\lambda}.$
 


Hazrat described the algebraic structure of Leavitt path algebras of finite acyclic graphs in \cite{r9}. Here we apply the description \cite[Theorem 4.14]{r9} to the case of connected finite acyclic graphs whose vertices have out-degree at most 1.

		\begin{lemma}
			\label{lemmmma 4.8}
			Let $E$ be a connected finite acyclic graphs whose vertices have out-degree at most 1 and $K$ a field. For the unique sink $v$ in $E^0$, let $\left\lbrace p_{i}|1\leq i \leq n\right\rbrace$ be the set of all paths which end in $v$. Then there is a $\mathbb  Z$-graded isomorphism 
						\begin{equation}
							\label{iso1}			
									L_{K}(E) \cong_{gr} M_{n}(K)(|p_{1}|,\dots,|p_{n}^v|)	
					\end{equation} sending $p_ip_j^{*}$ to $e_{ij}$ for each $1\leq i, j\leq n, 1\leq s \leq t$. Here $e_{ij}$ is the matrix with 1 in the $ij$-position and zero elsewhere.
					 
			Furthermore, let $F$ be an another connected finite acyclic graph whose vertices have out-degree at most 1. For the unique sink $u$ in $F^0$ let $\left\lbrace p_{i}'|1\leq i \leq n'\right\rbrace$ be the set of all paths which end in $u$. Then $L_{K}(E) \cong _{gr} L_{K}(F)$ if and only if $n= n'$ and after a permutation of indices,  $\left\lbrace |p_{i}||1\leq i \leq n\right\rbrace $ and $\left\lbrace |p_{i}'||1\leq i \leq n' \right\rbrace $ present the same list.
		\end{lemma}

			Let $K$ be a field and $K[x^{s},x^{-s}] = \bigoplus_{i\in s{\mathbb  Z}}Kx^{i}$ be the graded ring of Laurent polynomials $\mathbb  Z$-graded by $K[x^{s},x^{-s}]_{sk} = Kx^{sk} $ and $K[x^{s},x^{-s}]_{n} = 0 $ if $s$ does not divide $n$. The following consequence will be used in the proof of Lemma \ref{graded isomorphism}.

		\begin{lemma}
			\label{lem4.7}
			Let $K[x^{s},x^{-s}] = \bigoplus_{i\in s{\mathbb  Z}}Kx^{i}$ be the graded ring of Laurent polynomials. Then each invertible element in $K[x^{s},x^{-s}]$ is homogeneous, that is invertible elements in $K[x^{s},x^{-s}]$ only belong to $\cup_{j\in \mathbb  Z} Kx^{sj}$.
		\end{lemma}
		\begin{proof}
			Suppose that there exists an invertible element $k_{1}x^{si_{1}}+\dots+k_{j}x^{si_{j}} (i_{1}\textgreater \dots \textgreater i_{j})$  in $K[x^{s},x^{-s}]$ which is not homogeneous. Then there exists an element $f(x)\in K[x^{s},x^{-s}]$ such that $$f(x)(k_{1}x^{si_{1}}+\dots+k_{j}x^{si_{j}})=1.$$ By comparing the degrees of the two sides, the sum of the highest degrees of $f(x)$ and $k_{1}x^{si_{1}}+\dots+k_{j}x^{si_{j}}$ is zero, as 1 has degree zero. Then $f(x)$ has the highest degree $-si_{1}$. Similarly $f(x)$  has the lowest degree $-si_{j}$. But we have $-si_{j}\textgreater -si_{1}$. This is a contradiction. So each invertible element in $K[x^{s},x^{-s}]$ is homogeneous.
		\end{proof}

The following consequence explores the isomorphisms between two graded matrix rings.
		
		\begin{lemma} [{\cite{r9}}, {\cite[Lemma 2.1]{r14}}]
			\label{keylemma}
			Let $R$ be a $\Gamma$-graded ring and $\gamma_{1},\dots, \gamma_{n}\in \Gamma$ with $\Gamma$ an abelian group.
			
			(1) If $\pi$ is a permutation of set $\left\lbrace 1,\dots,n\right\rbrace $, then 
			\begin{equation}\label{1}
				M_{n}(R)(\gamma_{1}, \gamma_{2}, \dots, \gamma_{n})\cong_{gr} M_{n}(R)(\gamma_{\pi(1)}, \gamma_{\pi(2)}, \dots, \gamma_{\pi(n)})
			\end{equation}	
			by the map $x\mapsto pxp^{-1}$ where $p$ is the permutation matrix with 1 in the $(i,\pi(i))$-position for $i=1,\dots,n$ and zero elsewhere.
			
			(2) For any $\delta\in\Gamma$, we have
			\begin{equation}\label{2}
				M_{n}(R)(\gamma_{1}, \gamma_{2}, \dots, \gamma_{n})\cong_{gr} M_{n}(R)(\gamma_{1}+\delta, \gamma_{2}+\delta, \dots, \gamma_{n}+\delta)
			\end{equation} via the identity map as the isomorphism.	
			
			(3)  If $\delta$ is such that there is an invertible element $u_{\delta}$ in $R_{\delta}$, then 
			
			\begin{equation}\label{fff}
				M_{n}(R)(\gamma_{1}, \gamma_{2}, \dots, \gamma_{n})\cong_{gr} M_{n}(R)(\gamma_{1}+\delta, \gamma_{2}, \dots, \gamma_{n})
			\end{equation}	
			by the map $x\mapsto u^{-1}xu$ where $u$ is the diagonal matrix with $u_{\delta},1,1,\dots,1$ on the diagonal.

		If $\Gamma$ is abelian and $R$ and $S$ are $\Gamma$-graded division rings, then 
		$$M_{n}(R)(\gamma_{1},\gamma_{2},\dots,\gamma_{n})\cong_{gr}M_{m}(S)(\delta_{1},\delta_{2},\dots,\delta_{m}) $$
		implies that $R\cong_{gr}S$, that $m=n$, and the list $\delta_{1},\delta_{2},\dots,\delta_{m}$ is obtained from the list $\gamma_{1},\gamma_{2},\dots,\gamma_{n}$ by a composition of finitely many operations as in part (1) to (3).
	\end{lemma}


		We recall from \cite[Theorem 4.20]{r9} the algebraic structure of Leavitt path algebras of multi-headed comets. We apply the result to the case of connected finite graphs with a unique non-trivial cycle whose vertices have out-degree at most 1.
		
		\begin{lemma}
			\label{lemmmma 4.9}
			Let $K$ be a field and $E$ a connected finite graph with a non-trivial cycle $C$ of length $l$, whose vertices have out-degree at most 1. Choose $v$ (an arbitrary vertex) in $C$ and remove the edge $e$ with $s(e) = v$ from the cycle $C$. In this new graph, let $\left\lbrace p_{i}|1\leq i \leq n\right\rbrace$ be the set of all paths which end in $v$. Then there is a $\mathbb  Z$-graded isomorphism 
			\begin{equation}
				\label{iso2}
				L_{K}(E) \cong _{gr}  M_{n}(K[x^{l},x^{-l}])(|p_{1}|,\dots,|p_{n}|)
			\end{equation}	
			sending $p_iC_{l}^{k}p_j^{*} $ to $e_{ij}(x^{kl})$ for each $1\leq i,j \leq n,  k\in \mathbb  Z$. Here $e_{ij}(x^{kl})$ is the matrix with $x^{kl}$ in the $ij$-position and zero elsewhere.
			
			\indent  Furthermore, let  $F$ be an another connected finite graph with a non-trivial cycle $C'$ of length $l'$, whose vertices have out-degree at most 1. Choose $u$ (an arbitrary vertex) in $C'$ and remove the edge $\alpha'$ with $s(\alpha') = u$ from the cycle $C'$. In this new graph, let $\left\lbrace p_{i}'|1\leq i \leq n'\right\rbrace$ be the set of all paths which end in $u$. Then $L_{K}(E) \cong _{gr} L_{K}(F)$ if and only if  $l= l^{'}$, $n=n'$, and after a permutation of indices, and $\left\lbrace |p_{i}||1\leq i \leq n\right\rbrace$ can be obtained from $\left\lbrace |p_{i}'||1\leq i \leq n'\right\rbrace$ by a composition of finitely many operations as in parts \eqref{1} to \eqref{fff}. 
		\end{lemma}

\subsection{Main result} In this subsection, we give a combinatorial sufficient and necessary condition to classify Leavitt path algebras and  Leavitt inverse semigroups of graphs whose vertices have out-degree at most $1$ up to graded isomorphisms.

Let $E$ be a connected finite graph whose vertices have out-degree at most 1. If $E$ has a sink, then we consider the sink as a trivial cycle with length 0. If $E$ contains a cycle $C$ with length $s$, then we choose a vertex $v_0$ on $C$ and remove the edge (if any) starting from $v_0$. Thus we get a new graph $E'$. Given nonnegative integer $m$, we define $m (\rm mod ~0) \equiv m$. For any $v \in E^0$, if $p$ is the unique path from $v$ to $v_0$ in $E'$, we call $|p| ({\rm mod}~ s)$ the relative depth of $v$ with respect to $v_0$.

The following theorem is the main result.

\begin{theorem}
\label{theorem}
Let $E$ and $F$ be connected finite graphs whose vertices have out-degree at most 1 and $C_1$ and $C_2$ (possibly trivial) be cycles respectively in $E$ and $F$. The following three statements are equivalent.

(1) $L_{K}(E)\cong_{gr} L_{K}(F)$ as $\mathbb  Z$-graded Leavitt path algebras;

(2) $LI(E) \cong_{gr} LI(F)$ as $\mathbb{Z}$-graded Leavitt inverse semigroups;

(3) $|C_1| = |C_2|$ and there exist $v_0$ on $C_1$ and $w_0$ on $C_2$ such that for any $d \in \{0, 1, \ldots, |C_1| - 1\}$, the number of vertices in $E$ having relative depth $d$ with respect to $v_0$ is equal to the number of vertices in $F$ having relative depth $d$ with respect to $w_0$.
\end{theorem}

\begin{remark}
\label{remark}
\begin{itemize}
    \item[(1)] When $C_1$ and $C_2$ in Theorem \ref{theorem} (3) are trivial cycles, then they are sinks in $E$ and $F$ respectively.  The statement (3) turns to be that there exist a sink $v_0\in E^0$ and a sink $w_0\in F^0$  such that for any non-negative integer $d$, the number of vertices in $E$ having relative depth $d$ with respect to  $v_0$ is equal to the number of vertices in $F$ having relative depth $d$ with respect to $w_0$. Here the relative depth is exactly the length of the path starting from the vertex and ending at the sink.

    \item[(2)] For (2)$\Rightarrow$ (1) of Theorem \ref{theorem},	we suppose that $LI(E)\cong_{gr}LI(F)$.  The non-zero elements of $LI(E)$ are precisely the non-zero elements in a natural basis for $L_{K}(E)$. So a graded isomorphism between $LI(E)$ and $LI(F)$ is a bijection between the natural bases of $L_{K}(E)$ and $L_{K}(F)$ that also preserves multiplication of basis elements in the algebras and the grading. Hence it induces a graded isomorphism between $L_{K}(E)$ and $L_{K}(F)$. 
    \item[(3)]  For (1)$\Rightarrow$ (2) of Theorem \ref{theorem}, the idea for the proof is as follow: We observe that when $L_{K}(E) \cong_{\rm gr} L_{K}(F)$ we have the following graded isomorphism  \begin{align*}
            \phi:	L_{K}(E)&\stackrel{f}{\longrightarrow}  M_{m}(K[x^{l_{1}},x^{-l_{1}}])(|p_{1}|,|p_{2}|,\dots,|p_{m}|)\\
            &\stackrel{g}{\longrightarrow} M_{n}(K[x^{l_{2}},x^{-l_{2}}])(|q_{1}|,|q_{2}|,\dots,|q_{n}|)\\
            &\stackrel{h}{\longrightarrow} L_{K}(F),
        \end{align*}
 where $f$ and $h$ are given by \eqref{iso2} and $g$ is one of the three kinds of explicit isomorphisms given by Lemma \ref{keylemma}. We will prove that $\phi$ sends basis elements of $L_K(E)$ to basis elements of $L_K(F)$. Based on this, we obtain the induced graded isomorphisms between $LI(E)$ and $LI(F)$. The precise proof for (1)$\Rightarrow$ (2) is given as the proof of Lemma \ref{graded isomorphism} below.
\end{itemize}
    \end{remark}

\begin{lemma}\label{graded isomorphism}
    	Let $E$ and $F$ be connected finite graphs whose vertices have out-degree at most 1 and $K$ a field. Then the graded isomorphism  $L_{K}(E)\cong_{gr} L_{K}(F)$ as $\mathbb  Z$-graded algebras implies the graded isomorphism  $LI(E) \cong_{gr} LI(F)$ as $\mathbb  Z$-graded Leavitt inverse semigroups  .
\end{lemma}
\begin{proof}	
 	We assume that $E$ is a tree but $F$ is not a tree. It follows from \cite[Lemma 3.4]{r7} that the dimension of  $L_{K}(E)$  as $K$ vector space is finite. But it follows \cite[Theorem 1.6.21]{r2} that the dimension of $L_{K}(F)$ is infinite. Hence we have $ L_{K}(E)\ncong L_{K}(F)$. Then there are only two cases, that is either $E$ and $F$ are both trees, or $E$ and $F$ are neither.

			The first case is that $E$ and $F$ are both trees. Then $E$ and $F$ respectively has a sink, denoted by $v$ and $v^{'}$. Let $R(v)=\left\lbrace p_{1}, p_{2},\dots,p_{n}\right\rbrace $  be the set of all paths ending at $v$ and $R(v^{'})=\left\lbrace q_{1}, q_{2},\dots,q_{n^{'}}\right\rbrace $ the set of all paths ending at $v^{' }$. If $L_{K}(E)\cong _{gr} L_{K}(F)$, by Lemma \ref{lemmmma 4.8}, we have $n=n^{'}$ and  $\left| p_{i}\right| = \left| q_{i}\right| $, $i=1,2,\dots,n$.
			We have 	$$	\phi:	L_{K}(E)\stackrel{f}{\longrightarrow}  M_{n}(K)(|p_{1}|,|p_{2}|,\dots,|p_{m}|)\stackrel{g}{\longrightarrow}  L_{K}(F)$$
			which is the composition of these three graded isomorphisms $f$ and $g$ sending elements as follows:
				$$p_{i}p_{j}^{*}\longmapsto e_{ij}\longmapsto  q_{i}q_{j}^{*}.$$
				Here $e_{ij}$ is the matrix with 1 at $ij$-position and zero elsewhere, $f$ and $g$ are given by \eqref{iso1}.
				 
			Now we have an induced homomorphism 
			\begin{align*}
				\widetilde{\phi} : LI(E) &\longrightarrow LI(F)\\
				p_{i}p_{j}^{*} &\longmapsto \phi(	p_{i}p_{j}^{*}). 
			\end{align*} between Leavitt inverse semigroups. Note that  $\widetilde{\phi}$ is a graded semigroup isomorphism.

			The second case is that $E$ and $F$ are not trees. It follows from  \cite[Proposition 2.1]{r1} that $E$ and $F$ respectively has a non-trivial cycle, denoted by $C_{1}$ of length $l_{1}$ and $C_{2}$ of length $l_{2}$. We choose $v \in C_{1}$, $v'\in C_{2}$  and remove the edge $e$ with $s(e)=v$ and edge $e^{'}$ with $s(e^{'})=v^{'}$. In the two new graphs, let $\left\lbrace p_{1},\dots , p_{m}\right\rbrace $ be the set of all paths which end in $v$ and $\left\lbrace q_{1},\dots , q_{n} \right\rbrace $ be the set of all paths which end in $v^{'}$.
			
			If $L_{K}(E)\cong_{gr}L_{K}(F)$, then by Lemma \ref{lemmmma 4.9}, we have $ l_{1}=l_{2}$ and $m=n$ and $\left\lbrace |p_{i}^{v_{s}}||1\leq i \leq n(v_{s})\right\rbrace$ can be obtained from $\left\lbrace |p_{i}^{u_{s}}||1\leq i \leq n(u_{s})\right\rbrace$ by a composition of finitely many operations as in parts \eqref{1} to \eqref{fff}. Set $\gamma_i=|p_i|$ for $1\leq i\leq n$.

			(1) For the first case, we have $|q_i|=\gamma_{\pi(i)}$ for $1\leq i\leq n$ and $\pi$ some permutation of the set $\{1,2,\cdots, n\}$.  
			
		\noindent Take any $p_{i}C_{1}^{k}p_{j}^{*}\in L_{K}(E)$ with $1\leq i,j\leq n$. We can assume that $\pi(s)=i \text{ and } \pi(t)=j$, where $1\leq s,t \leq n$. We have	
     \begin{align*}
            \phi:	L_{K}(E)&\stackrel{f}{\longrightarrow}  M_{m}(K[x^{l_{1}},x^{-l_{1}}])(|p_{1}|,|p_{2}|,\dots,|p_{m}|)\\
            &\stackrel{g}{\longrightarrow} M_{n}(K[x^{l_{2}},x^{-l_{2}}])(|q_{1}|,|q_{2}|,\dots,|q_{n}|)\\
            &\stackrel{h}{\longrightarrow} L_{K}(F)
        \end{align*}
			  which is the composition of these three graded isomorphisms $f$, $g$ and $h$, sending elements as follows:
			$$p_{i}C_{1}^{k}p_{j}^{*}\longmapsto e_{ij}(x^{kl_{1}})\longmapsto e_{st}(x^{kl_{1}})\longmapsto q_{s}C_{2}^{k}q_{t}^{*}.$$
		Here $e_{ij}(x^{kl_1})$ is the matrix with $x^{kl_1}$ at $ij$-position and zero elsewhere, $f$ and $h$ are given by \eqref{iso2}. And for $g$ we use \eqref{1} to obtain that $g(x) = pxp^{-1}$, where $p$ is the permutation matrix with 1 at the $(i,\pi(i))$-position for $i=1,\dots,n$ and zeroes elsewhere. Hence we have $$g(e_{ij}(x^{kl_1}))=pe_{ij}(x^{kl_1})p^{-1}=e_{st}(x^{kl_1}).$$  
	
	  Now we have an induced homomorphism 
			\begin{align*}
				\widetilde{\phi} : LI(E) &\longrightarrow LI(F)\\
				p_{i}C_{1}^{k}p_{j}^{*} &\longmapsto \phi(	p_{i}C_{1}^{k}p_{j}^{*}). 
			\end{align*} between Leavitt inverse semigroups. Note that  $\widetilde{\phi}$ is a graded semigroup isomorphism.

		(2)	For the second case, we have $|q_{i}| = \gamma_{i}+\alpha$ for $1\leq i 
		\leq n$, where $\alpha\in \mathbb  Z$. Take any $p_{i}C_{1}p_{j}^*\in L_{K}(E)$, where $1\leq i,j \leq n$. We have  \begin{align*}
            \phi:	L_{K}(E)&\stackrel{f}{\longrightarrow}  M_{m}(K[x^{l_{1}},x^{-l_{1}}])(|p_{1}|,|p_{2}|,\dots,|p_{m}|)\\
            &\stackrel{g}{\longrightarrow} M_{n}(K[x^{l_{2}},x^{-l_{2}}])(|q_{1}|,|q_{2}|,\dots,|q_{n}|)\\
            &\stackrel{h}{\longrightarrow} L_{K}(F),
        \end{align*}
		  which is the composition of these three graded isomorphisms $f$, $g$ and $h$, sending elements as follows:
		$$p_{i}C_{1}^{k}p_{j}^{*}\longmapsto e_{ij}(x^{kl_{1}})\longmapsto e_{ij}(x^{kl_{1}})\longmapsto q_{i}C_{2}^{k}q_{j}^{*}.$$
		Here $e_{ij}(x^{kl_1})$ is the matrix with $x^{kl_1}$ at $ij$-position and zero elsewhere, $f$ and $h$ are given by \eqref{iso2}. And for $g$ we use \eqref{2} to obtain the following formula : $$g(e_{ij}(x^{kl_1}))=e_{ij}(x^{kl_1}).$$  
		
		Now we have an induced homomorphism 
		\begin{align*}
			\widetilde{\phi} : LI(E) &\longrightarrow LI(F)\\
			p_{i}C_{1}^{k}p_{j}^{*} &\longmapsto \phi(	p_{i}C_{1}^{k}p_{j}^{*}). 
		\end{align*} between Leavitt inverse semigroups. Note that  $\widetilde{\phi}$ is a graded semigroup isomorphism.

		(3) For the third case, without loss of generality we can assume that $|q_{1}| = \gamma_{1}+\delta, |q_{i}| = \gamma_{i}$ for $2\leq i \leq n$, where $\delta = \lambda l_{1}, \lambda\in \mathbb  Z$; refer to Lemma \ref{lem4.7}.

Take any $p_{i}C_{1}p_{j}^{*}\in L_{K}(E)$, where $1\leq i,j \leq n$. We have 
        \begin{align*}
            \phi:	L_{K}(E)&\stackrel{f}{\longrightarrow}  M_{m}(K[x^{l_{1}},x^{-l_{1}}])(|p_{1}|,|p_{2}|,\dots,|p_{m}|)\\
            &\stackrel{g}{\longrightarrow} M_{n}(K[x^{l_{2}},x^{-l_{2}}])(|q_{1}|,|q_{2}|,\dots,|q_{n}|)\\
            &\stackrel{h}{\longrightarrow} L_{K}(F)
        \end{align*}
	  which is the composition of these three graded isomorphisms $f$, $g$ and $h$, sending elements as follows:
		
		\textbf{case 1}: $i = j =1$,
		$$p_{1}C_{1}^{k}p_{1}^{*}\longmapsto e_{11}(x^{kl_{1}})\longmapsto e_{11}(x^{kl_{2}}) \longmapsto q_{1}C_{2}^{k}q_{1}^{*},  $$
		
		\textbf{case 2}: $i = 1, j\neq 1$,
		$$p_{1}C_{1}^{k}p_{j}^{*}\longmapsto e_{1j}(x^{kl_{1}})\longmapsto e_{1j}(x^{kl_{2}-\delta}) \longmapsto q_{1}C_{2}^{k-\lambda}q_{j}^{*},$$
		
		\textbf{case 3}: $i \neq 1, j= 1$,
		$$p_{i}C_{1}^{k}p_{1}^{*}\longmapsto e_{i1}(x^{kl_{1}})\longmapsto e_{i1}(x^{kl_{2}+\delta}) \longmapsto q_{i}C_{2}^{k+\lambda}q_{1}^{*},$$
		
		\textbf{case 4}: $i \neq 1, j\neq 1$,
		$$p_{i}C_{1}^{k}p_{j}^{*}\longmapsto e_{ij}(x^{kl_{1}})\longmapsto e_{ij}(x^{kl_{2}}) \longmapsto q_{i}C_{2}^{k}q_{j}^{*}.$$
		Here $e_{ij}(x^{kl_1})$ is the matrix with $x^{kl_1}$ at $ij$-position and zero elsewhere, $f$ and $h$ are given by \eqref{iso2}. And for $g$ we use \eqref{fff} to obtain that $g(x)=u^{-1}x u$, where $u$ is the diagonal matrix with $x^{\delta}, 1, 1,\dots,1$ on the diagonal. Hence for $i = j=1$ we have 	$$g(e_{11}(x^{kl_{1}})) = u^{-1}e_{11}(x^{kl_{1}})u = e_{11}(x^{kl_{1}}).$$
		
  For $i = 1, j\neq 1$, we have
		$$g(e_{1j}(x^{kl_{1}})) = u^{-1}e_{1j}(x^{kl_{1}})u = e_{1j}(x^{kl_{1}-\delta}).$$
		
		For $i \neq 1, j = 1$, we have
		$$g(e_{i1}(x^{kl_{1}})) = u^{-1}e_{i1}(x^{kl_{1}})u = e_{i1}(x^{kl_{1}+\delta}),$$
		
		For $i \neq 1, j \neq 1$, we have
		$$g(e_{ij}(x^{kl_{1}})) = u^{-1}e_{ij}(x^{kl_{1}})u = e_{ij}(x^{kl_{1}}).$$

Now we have an induced homomorphism  
		\begin{align*}
			\widetilde{\phi} : LI(E) &\longrightarrow LI(F)\\
			p_{i}C_{1}^{k}p_{j}^{*} &\longmapsto \phi(	p_{i}C_{1}^{k}p_{j}^{*}). 
		\end{align*} between Leavitt inverse semigroups. Note that  $\widetilde{\phi}$ is a graded semigroup isomorphism. Therefore we have that the isomorphism $L_{K}(E)\cong_{gr} L_{K}(F)$ implies $LI(E)\cong_{gr} LI(F).$
  \end{proof}

As the preparation for the proof of Theorem \ref{theorem} we recall the following result on Leavitt inverse semigroups. We first recall that a path $p = e_{1}e_{2}\dots e_{n}$ in $E$ has  exits if at least one of the vertices $s(e_{i})$ has out-degree greater than 1. In particular, an edge $e\in E^1$ has exits if and only if $s(e)$ has out-degree greater than 1. We say that $p=e_1e_2\dots e_n$ is an NE path if every vertex $s(e_i)$, $i=1,\dots, n$ has out-degree 1. We also define the empty path at any vertex $v$ to be an NE path.
 
By \cite[Lemma 4.9 (a) (b)]{r1} and \cite[Lemma 6.6 (e)]{r1} we have the following consequence.

\begin{lemma}\label{equalequiv}
Let $E$ and $F$ be two directed connected finite graphs whose vertices having out-degree at most $1$. If $\phi$ is an isomorphism from $LI(E)$
onto $LI(F)$, then the following statements hold.

(a) The isomorphism $\phi$ induces a bijection from $E^0$ onto $F^0$.

(b) For any nonzero $p q^* \in LI(E)$, if $\phi(p q^*) = p_1
q_1^*$ and $p, q$ are NE paths, then $p_1, q_1$ are NE paths,
$\phi(s(p)) = s(p_1)$ and $\phi(s(q)) = s(q_1)$;

(c) If $C$ is a cycle in $E$, then $\phi(C)$ is uniquely
expressible in the form $\phi(C) = pC'p^*$ or $\phi(C) = pC'^*p^*$
in $LI(F)$ for some cycle $C'$ and some NE path $p$ in $F$, and moreover $\phi^{-1}(C') = p_1 C_1 p_1^*$ or $\phi^{-1}(C') = p_1 C_1^* p_1^*$ for some cyclic conjugate $C_1$ of $C$ and some NE path $p_1$ in $E$.
\end{lemma}

\noindent{\bf Proof of Theorem \ref{theorem}:} The equivalence of (1) and (2) in Theorem \ref{theorem} follows immediately from Remark \ref{remark} (2) and Lemma \ref{graded isomorphism}. It remains to show (3)$\Rightarrow$ (1) and (2)$\Rightarrow$ (3).

(3)$\Rightarrow$ (1): First assume that $C_1$ and $C_2$ are nontrivial cycles. We remove the edge starting from $v_0$. Thus we get a new graph $E'$. Similarly we remove the edge starting from $w_0$ and get a new graph $F'$.  We suppose that $|C_1| = |C_2|=s$. Set  $$X_d=\{ p ~|~  p\in \text{Path }(E') \text{~with ~} r(p)=v_0, |p|({\rm mod} ~s) \equiv d \}.$$  We list $p_1^d, p_2^d,\cdots, p_{|X_d|}^d$ as all the elements in $X_d$. Then by Lemma \ref{lemmmma 4.9},
we have 
$$L_K(E)\cong _{gr} M_n(K[x^s,x^{-s}])(\underbrace{|p^0_1|,\dots,|p^0_{|X_0|}|}_{|X_0|},\underbrace{|p^1_1|,\dots,|p^1_{|X_1|}|}_{|X_1|},\dots,\underbrace{|p^{s-1}_1|,\dots, |p^{s-1}_{|X_{s-1}|}|}_{|X_{s-1}|})$$  with $n= \sum \limits_{d=0} ^{s-1}|X_d|$. 

Similarly, we set $$Y_d=\{ q ~|~  q\in \text{Path }(F') \text{~with ~} r(q)=w_0, |q|({\rm mod} ~s) \equiv d \}$$  and  list $q_1^d, q_2^d,\cdots, q_{|Y_d|}^d$ as all the elements in $Y_d$. Since $|X_d| = |Y_d|$ for each $d$, we have $n=\sum \limits_{d=0} ^{s-1}|Y_d|$. Then by Lemma \ref{lemmmma 4.9}, we have 
$$L_K(F)\cong _{gr} M_n(K[x^s,x^{-s}])(\underbrace{|q^0_1|,\dots,|q^0_{|Y_0|}|}_{|Y_0|},\underbrace{|q^1_1|,\dots,|q^1_{|Y_1|}|}_{|Y_1|},\dots,\underbrace{|q^{s-1}_1|,\dots, |q^{s-1}_{|Y_{s-1}|}|}_{|Y_{s-1}|}).$$
For each $d\in \{ 0,1,\dots,s-1\}$ and each $ 1\leq i \leq |X_d|$, we have $|p_i^d|=|q_i^d|+ks$ for some $k\in \mathbb{Z}$. By Lemma \ref{keylemma} (3), we have $L_K(E)\cong _{gr}L_K(F).$

When $C_1$ and $C_2$ are trivial cycles, then they are sinks. We suppose that $v_0$ and $w_0$ be sinks respectively in $E$ and $F$. Set $$X'_d = \{ p ~|~  p\in \text{Path }(E) \text{~with ~} r(p)=v_0, |p|= d \}.$$ We list $p_1^d, p_2^d,\cdots, p_{|X'_d|}^d$ as all the elements in $X'_d$. Then by Lemma \ref{lemmmma 4.8},
we have 
$$L_K(E)\cong _{gr} M_n(K)(\underbrace{|p^0_1|,\dots,|p^0_{|X'_0|}|}_{|X'_0|},\underbrace{|p^1_1|,\dots,|p^1_{|X'_1|}|}_{|X'_1|},\dots,\underbrace{|p^{s-1}_1|,\dots, |p^{s-1}_{|X'_{s-1}|}|}_{|X'_{s-1}|})$$  with $n= \sum \limits_{d=0} ^{s-1}|X'_d|$. Similarly, we denote the set $Y'_d = \{ q ~|~  q\in \text{Path }(F) \text{~with ~} r(q)=w_0, |q|= d \}.$  We list $q_1^d, q_2^d,\cdots, q_{|Y'_d|}^d$ as all the elements in $Y'_d$. So we have $|X'_d| = |Y'_d|$ for each $d$. Then by Lemma \ref{lemmmma 4.8}, we have 
$$L_K(F)\cong _{gr} M_n(K)(\underbrace{|q^0_1|,\dots,|q^0_{|Y'_0|}|}_{|Y'_0|},\underbrace{|q^1_1|,\dots,|q^1_{|Y'_1|}|}_{|Y'_1|},\dots,\underbrace{|q^{s-1}_1|,\dots, |q^{s-1}_{|Y'_{s-1}|}|}_{|Y'_{s-1}|}).$$
For each $d$ and each $ i$, we have $|p_i^d|=|q_i^d|$. Hence, we have $L_K(E)\cong _{gr}L_K(F).$

(2)$\Rightarrow$ (3): Let $\phi$ be a graded isomorphism from $LI(E)$ to $LI(F)$. It follows from Lemmas \ref{lemma 11} and \ref{equalequiv}(c) that $\phi(C_1) = q C_2 q^*$ for some NE path $q$ in $F$ since $\phi$ is graded. So we have $|C_1| = |C_2|$ (Note that this is also true if $E$ contains no cycles). To prove the remaining part, it suffices to show that $\phi$ preserves the relative depth of vertices for appropriately chosen vertices respectively on $C_1$ and $C_2$. 

Fix a vertex $u_0$ on $C_1$ and assume that $\phi(u_0) = v_1$ for some $v_1 \in F^0$ by using Lemma~\ref{equalequiv}(a).  If $C_1$ is trivial, then $u_0$ is the unique sink in $E$. In this case $C_2$ must be trivial. We choose the unique sink of $C_2$ as $v_0$ and suppose that $q$ is the unique path from $v_1$ to $v_0$. Let $\phi^{-1}(v_0) = u_1$ the $p$ be the unique path from $u_1$ to $u_0$. Then we see from Lemma~\ref{equalequiv}(b) that $$\phi(p) = \phi(u_1pu_0) = \phi(u_1) \phi(p) \phi(u_0) = \phi(u_1) \phi(p) \phi(u_0) = v_0 \phi(p) v_1 q^*.$$ Thus we observe that both $p$ and $q$ must be empty since $\phi$ is graded. This leads to $\phi(u_0) = v_0$. If $C_1$ is nontrivial. It is easy to choose a vertex $v_0$ on $C_2$ such that the unique path $q$ from $v_1$ to $v_0$  has length $m|C_1|$ for some nonnegative integer $m$. 

Now for any $u \in E^0$, suppose that $\phi(u) = v$ for some $v \in F_0$ and $p_1$ is the unique path from $u$ to $u_0$ which does not contain the cycle $C_2$. Then again we see from Lemma~\ref{equalequiv}(b) that $$\phi(p_1) =\phi(up_1u_0) =\phi(u)\phi(p_1)\phi(u_0)=vq_1 C_2^l q^*v_1=q_1 C_2^l q^*$$ for some $l \in \mathbb{Z}$ and NE path $q_1$ with $s(q_1) = v$ and $r(q_1) = v_0$. Therefore, $|p_1| = |q_1| + l |C_2| - |q|$ since $\phi$ is graded. If $C_1$ is trivial, then $|p_1| = |q_1|$. If $C_1$ is nontrivial, then by noticing that $|q|$ is a multiple of $|C_1| = |C_2|$, we observe that the relative depth of $u$ with respect to $u_0$ also equals to the relative depth of $v$ with respect to $v_0$. Equivalently, in both cases, we have $$|p_1| ({\rm mod} ~|C_2|)=|q_1|  ({\rm mod} ~|C_1|).$$ The proof of Theorem \ref{theorem} is completed.

We give an example of two graphs below which satisfy the combinatorial condition of Theorem  \ref{theorem}.

\begin{example} Let $F_{1}$ and $F_{2}$ be the following graphs.
			\begin{equation*}
				{\def\labelstyle{\displaystyle}
					\xymatrix{	
					F_1: &v_{1} \ar[r]^{e_{1}} & v_{2} \ar@/^/[r]^{e_{2}} &v_{3} \ar@/^/[l]^{e_{3}} & v_{4} \ar[l] ^{e_{4}} 
				} \qquad \qquad
			}		
			{\def\labelstyle{\displaystyle}
			\xymatrix{	
				F_2: &w_{2} \ar@/^/[r]^{f_{1}} & w_{3} \ar@/^/[l]^{f_{2}} & w_{4} \ar[l] ^{f_{3}} &w_{1} \ar[l]^{f_{4}} 
			} 
		}			
			\end{equation*}
			We choose $v_{3}$ in $F_{1}$ and remove the edge $e_{3}$ with $s(e_{3}) = v_{3}$. In the new graph $F_1'$, let $\left\lbrace v_{3}, e_{2}, e_{4}, e_{1}e_{2}\right\rbrace $ be the set of all paths sending at $v_{3}$ with $p_{1} =v_{3}, p_{2} = e_{2},p_3=e_4, p_{4}= e_{1}e_{2}$. Then $L_{K}(F_{1})\cong _{gr} M_{4}(K[x^{2},x^{-2}])(0,1,1,2).$ Similarly, we choose $w_{2}$ in $F_{2}$ and remove the edge $f_{1}$ with $s(f_{1}) = w_{2}$. In the new graph $F_2'$, let $\left\lbrace w_{2}, f_{2}, f_{3}f_{2},  f_{4}f_{3}f_{2} \right\rbrace $ be the set of all paths ending at $w_{2}$ with $q_{1} = w_{2}, q_{2} = f_{2}, q_3=f_3f_2, q_{4} = f_{4}f_{3}f_{2}$. Then $L_{K}(F_{2}) \cong _{gr} M_{4}(K[x^{2},x^{-2}])(0,1,2,3).$
			By Lemma \ref{keylemma}, we have 
		\begin{align*}
				M_{4}(K[x^{2},x^{-2}])(0,1,1,2) &\cong_{gr} M_{4}(K[x^{2},x^{-2}])(0+1,1+1,1+1,2+1)\\
				& \cong_{gr} M_{4}(K[x^{2},x^{-2}])(1,2-2,2,3)\\
				& \cong_{gr} M_{4}(K[x^{2},x^{-2}])(0,1,2,3).
			\end{align*}

   Set  $C_1=e_3e_2$ and $C_2=f_1f_2$ to be cycles in $E$ and $F$ respectively. 
   Take any $p_{i}C_{1}^{k}p_{j}^{*}\in L_{K}(F_{1})$, where $1\leq i,j\leq 4$ and $k\in \mathbb{Z}$. We have 			\begin{align*}
				\phi: L_{K}(F_{1})&\stackrel{f}{\longrightarrow} M_{4}(K[x^{2},x^{-2}])(0,1,1,2)\\
				& \stackrel{g}{\longrightarrow} M_{4}(K[x^{2},x^{-2}])(0+1,1+1,1+1,2+1)\\
				& \stackrel{h}{\longrightarrow} M_{4}(K[x^{2},x^{-2}])(1,2-2,2,3)\\
				&\stackrel{\eta}{\longrightarrow} M_{4}(K[x^{2},x^{-2}])(0,1,2,3)\\
				&\stackrel{\theta}{\longrightarrow} L_{K}(F_{2})
			\end{align*}
		which is the composition of five graded isomorphisms $f$, $g$, $h$, $\eta$ and $\theta$, sending elements as follows:

        \[ \begin{cases}
            p_{2}C_{1}^{k}p_{2}^{*}\longmapsto e_{22}(x^{2k})\longmapsto e_{22}(x^{2k})\longmapsto e_{22}(x^{2k})\longmapsto e_{11}(x^{2k})\longmapsto q_{1}C_{2}^{k}q_{1}^{*};\\

			p_{2}C_{1}^{k}p_{1}^{*}\longmapsto e_{21}(x^{2k})\longmapsto e_{21}(x^{2k})\longmapsto e_{21}(x^{2k+2})\longmapsto e_{12}(x^{2k+2})\longmapsto q_{1}C_{2}^{k+1}q_{2}^{*};\\
			p_{2}C_{1}^{k}p_{j}^{*}\longmapsto e_{2j}(x^{2k})\longmapsto e_{2j}(x^{2k})\longmapsto e_{2j}(x^{2k+2})\longmapsto e_{1j}(x^{2k})\longmapsto q_{1}C_{2}^{k+1}q_{j}^{*}, &   \quad j=3,4;\\

   p_{1}C_{1}^{k}p_{2}^{*}\longmapsto e_{12}(x^{2k})\longmapsto e_{12}(x^{2k})\longmapsto e_{12}(x^{2k-2})\longmapsto e_{21}(x^{2k-2})\longmapsto q_{2}C_{2}^{k-1}q_{1}^{*}; \\
				p_{i}C_{1}^{k}p_{2}^{*}\longmapsto e_{i2}(x^{2k})\longmapsto e_{i2}(x^{2k})\longmapsto e_{i2}(x^{2k-2})\longmapsto e_{i1}(x^{2k})\longmapsto q_{i}C_{2}^{k-1}q_{1}^{*}, &  \quad i=3,4;\\

    p_{1}C_{1}^{k}p_{1}^{*}\longmapsto e_{11}(x^{2k})\longmapsto e_{11}(x^{2k})\longmapsto e_{11}(x^{2k})\longmapsto e_{22}(x^{2k})\longmapsto q_{2}C_{2}^{k}q_{2}^{*}; \\
			p_{1}C_{1}^{k}p_{j}^{*}\longmapsto e_{1j}(x^{2k})\longmapsto e_{1j}(x^{2k})\longmapsto e_{1j}(x^{2k})\longmapsto e_{2j}(x^{2k})\longmapsto q_{2}C_{2}^{k-1}q_{j}^{*}, &   \quad j = 3,4;\\

    p_{i}C_{1}^{k}p_{1}^{*}\longmapsto e_{i1}(x^{2k})\longmapsto e_{i1}(x^{2k})\longmapsto e_{i1}(x^{2k})\longmapsto e_{i2}(x^{2k})\longmapsto q_{i}C_{2}^{k}q_{2}^{*}, & \quad i= 3,4;\\
    
   p_{i}C_{1}^{k}p_{j}^{*}\longmapsto e_{ij}(x^{2k})\longmapsto e_{ij}(x^{2k})\longmapsto e_{ij}(x^{2k})\longmapsto e_{ij}(x^{2k})\longmapsto q_{i}C_{2}^{k}q_{j}^{*},  &\quad  i,j= 3,4.
		\end{cases} \]
 Here $f$ and $\theta$ are given by \eqref{iso2}, $g$ is given by \eqref{2} , $h$ is given by \eqref{fff} and $\eta$ is given by \eqref{1}.

			Now we have an induced homomorphism 
			\begin{align*}
				\widetilde{\phi} : LI(F_{1}) &\longrightarrow LI(F_{2})\\
				p_{i}C_{1}^{k}p_{j}^{*} &\longmapsto \phi(	p_{i}C_{1}^{k}p_{j}^{*}). 
			\end{align*} between Leavitt inverse semigroups. Note that  $\widetilde{\phi}$ is a graded semigroup isomorphism.

The graphs $F_1$ and $F_2$ satisfy the combinatorial condition of Theorem  \ref{theorem}. The number of vertices in $F_1$ having relative depth $d$ ($d=0$ or $d=1$) with respect to 
$v_3$ is equal to the number of vertices in $F_2$ having relative depth $d$ with respect to $w_2$, as we have \begin{align*}|p_1| ({\rm mod }~ 2) \equiv 0, |p_2| ({\rm mod}~ 2) \equiv 1, |p_3| ({\rm mod}~ 2) \equiv 1, \text{~and~} |p_4| ({\rm mod }~ 2) \equiv 0;\\
|q_1| ({\rm mod }~ 2) \equiv 0, |q_2| ({\rm mod}~ 2) \equiv 1, |q_3| ({\rm mod}~ 2) \equiv 0, \text{~and~} |q_4| ({\rm mod }~ 2) \equiv 1.\end{align*}
		\end{example}

We give an example of two graphs below which do not satisfy the combinatorial condition of Theorem  \ref{theorem}.

\begin{example} Let $G_{1}$ and $G_{2}$  be the following two graphs.
			\begin{equation*}
				{\def\labelstyle{\displaystyle}
					
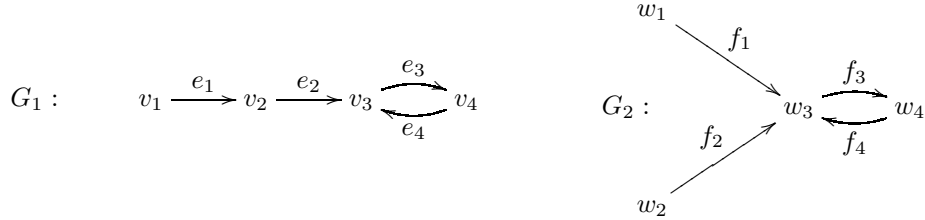
\xymatrix{	\\
						G_1:&v_{1} \ar[r]^{e_{1}}& v_{2} \ar[r]^{e_{2}}& v_{3} \ar@/^/[r]^{e_3}  &v_4 \ar@/^/[l]^{e_4}\\			
				}}	
				\qquad \qquad
				{\def\labelstyle{\displaystyle}
					\xymatrix{	
						      	w_1 \ar[dr]^{f_{1}} & &\\
						      G_2:\qquad &w_3 \ar@/^/[r]^{f_3} & w_4 \ar@/^/[l]^{f_4}\\
							w_{2} \ar[ur]^{f_{2}} &&
				} }			
			\end{equation*}

The two graphs $G_1$ and $G_2$ do not satisfy the combinatorial condition of Theorem \ref{theorem} (3). Actually we have the following observation.
  Suppose that we choose $v_{3}$ in $G_{1}$ with the cycle $C_1=e_3e_4$ and remove the edge $e_{3}$. In the new graph, let $\left\lbrace v_{3}, e_{2}, e_4, e_{1}e_{2}\right\rbrace $ be the set of all path ending at $v_{3}$ with $p_{1} =v_{3}, p_{2} = e_{2}, p_3= e_4, p_{4}= e_{1}e_{2}$. By Lemma \ref{lemmmma 4.9}, we have $$L_K(G_1)\cong _{gr} M_4(K[x^2,x^{-2}])(0,1,1,2)$$
 and that $$|p_1| ({\rm mod }~ 2) \equiv 0, |p_2| ({\rm mod}~ 2) \equiv 1, |p_3| ({\rm mod}~ 2) \equiv 1, \text{~and~} |p_4| ({\rm mod }~ 2) \equiv 0.$$ Then the number of vertices in $G_1$ having relative depth $0$ with respect to $v_3$ is $2$ and  the number of vertices in $G_1$ having relative depth $1$ with respect to $v_3$ is $2$. Suppose that we choose $v_{4}$ in $G_{1}$ with the cycle $C_1=e_3e_4$ and remove the edge $e_{4}$. Similarly we have  $$L_K(G_1)\cong _{gr} M_4(K[x^2,x^{-2}])(0,1,2,3)$$
and that the number of vertices in $G_1$ having relative depth $0$ with respect to $v_4$ is $2$ and  the number of vertices in $G_1$ having relative depth $1$ with respect to $v_4$ is $2$.

For the graph $G_2$,
suppose that we choose $w_{3}$ in $G_{2}$ with the cycle $C_2=f_3f_4$ and remove the edge $f_{3}$. We have $L_K(G_2)\cong _{gr} M_4(K[x^2,x^{-2}])(0,1,1,1)$
 and that the number of vertices in $G_2$ having relative depth $0$ with respect to $w_3$ is $1$ and  the number of vertices in $G_2$ having relative depth $2$ with respect to $w_3$ is $3$. Suppose that we choose $w_{4}$ in $G_{2}$ with the cycle $C_2=f_3f_4$ and remove the edge $f_{3}$. We have $L_K(G_2)\cong _{gr} M_4(K[x^2,x^{-2}])(0,1,2,2)$
 and that the number of vertices in $G_2$ having relative depth $0$ with respect to $w_4$ is $3$ and  the number of vertices in $G_2$ having relative depth $2$ with respect to $w_4$ is $1$. Therefore Theorem \ref{theorem}(3) does not hold for $G_1$ and $G_2$. 

For algebraic structure we can prove $$L_K(G_1)\ncong _{gr} L_K(G_2).$$ 
 By Lemma \ref{keylemma}, we have 
$$L_K(G_1) \cong_{gr} M_{4}(K[x^{2},x^{-2}])(0,1,1,2)\\
				 \cong_{gr} M_{4}(K[x^{2},x^{-2}])(0,1,2,3)\\
				 \cong_{gr} M_{4}(K[x^{2},x^{-2}])(0,0,1,1)$$
	and $$	L_K(G_2) \cong_{gr} M_{4}(K[x^{2},x^{-2}])(0,1,2,2)\\
				\cong_{gr} M_{4}(K[x^{2},x^{-2}])(0,1,1,1).$$
	\cite[Lemma 3.1(2)]{r14} said that if the elements $\gamma_1, \gamma_2, \gamma_3, \gamma_4$ are considered modulo $2$ and arranged in a nondecreasing order, the resulting list is $l_0(0), l_1(1)$ for some nonnegative integers $l_0, l_1$ such that $4=l_0+l_1$. The integers $l_0, l_1$ are unique for the graded isomorphism class of $$M_4(K[x^2, x^{-2}])(\gamma_1, \gamma_2,\gamma_3, \gamma_4)$$ up to their order. Here $M_{4}(K[x^{2},x^{-2}])(0,0,1,1)$ and  $M_{4}(K[x^{2},x^{-2}])(0,1,1,1)$ are arranged in this way. But the elements $0,0,1,1$ and $0, 1,1,1$ are different. Therefore  $L_K(E_1)\ncong _{gr} L_K(E_2)$.
	\end{example}

\begin{corollary}
Let $E$ and $F$ be connected finite graphs whose vertices have out-degree at most 1. If $C_1$ and $C_2$ are cycles having length $1$ respectively in $E$ and $F$, then the following five statements are equivalent. 

(1) $L_{K}(E)\cong_{gr} L_{K}(F)$ as $\mathbb  Z$-graded Leavitt path algebras;

(2) $L_{K}(E)\cong L_{K}(F)$ as non-graded Leavitt path algebras;

(3) $LI(E)\cong LI(F)$ as non-graded Leavitt inverse semigroups;

(4)  $LI(E)\cong_{gr} LI(F)$ as $\mathbb  Z$-graded Leavitt inverse semigroups;

(5) $|E^0|=|F^0|$.
\end{corollary}
\begin{proof}
    The equivalences of (1), (4) and (5) follow immediately from Theorem \ref{theorem}. Obviously  (1) implies (2). It follows from \cite[Theorem 4.7]{r1} that (2) and (3) are equivalent and (3) implies (5). The proof is completed. 
\end{proof}

In Example \ref{cexample} below the two graphs $F$ and $F'$ both have a vertex whose out-degree is $2$. We will show that the graded isomorphism of algebras $L_K(F)\cong_{\rm gr}L_K(F')$ does not imply $LI(F)\cong_{\rm gr}LI(F')$.

		\begin{example}
  \label{cexample}
  Consider the following graphs.
			\begin{equation*}
				{\def\labelstyle{\displaystyle}
					
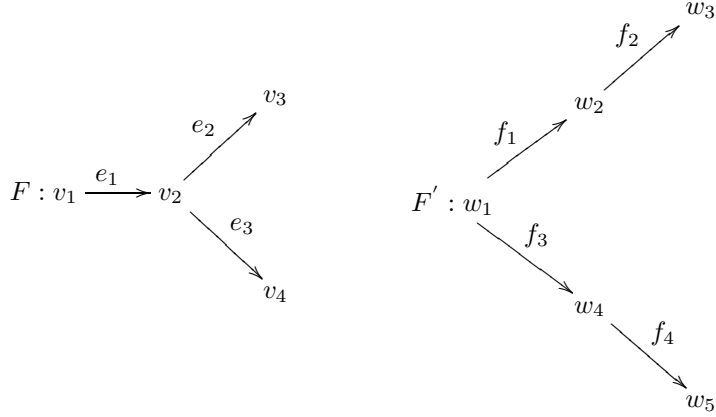
\xymatrix{	\\
						& & v_{3}\\
						F		 : v_{1}\ar[r]^{e_{1}} &v_{2}\ar[ur]^{e_{2}} \ar[dr]^{e_{3}}&\\
						& & v_{4}\\
						\\
					} \qquad \qquad
				}
				{\def\labelstyle{\displaystyle}
					\xymatrix{	
						& & w_{3}\\
						& w_{2}\ar[ur]^{f_{2}}&\\
						F^{'}:	w_{1}\ar[ur]^{f_{1}} \ar[dr]^{f_{3}}&&\\
						& w_{4}\ar[dr]^{f_{4}}&\\
						&&w_{5}	\\		
					}
				} 						
			\end{equation*}		
			Referring to \cite[Theorem 4.14]{r9} and comparing with Lemma \ref{lemmmma 4.8}, we have that $$L_{K}(F) \cong _{gr} L_{K}(F^{'}) \cong_{gr}  M_{3}(K)(0,1,2) \bigoplus M_{3}(K)(0,1,2). $$ By Example $\ref{example1}$, the Leavitt inverse semigroup is  $$LI(F)= \left\lbrace 0,v_{1}, v_{2}, e_{1}^{*}, v_{3}, e_{2}^{*}, e_{2}^{*}e_{1}^{*}, v_{4}, e_{3}^{*}, e_{3}^{*}e_{1}^{*}, e_{1}, e_{2}, e_{2}e_{2}^{*}, e_{2}e_{2}^{*}e_{1}^{*}, e_{3}, 
			\atop e_{3}e_{3}^{*}, e_{3}e_{3}^{*}e_{1}^{*}, e_{1}e_{2}, e_{1}e_{2}e_{2}^{*}, e_{1}e_{2}e_{2}^{*}e_{1}^{*}, e_{1}e_{3}, e_{1}e_{3}e_{3}^{*}, e_{1}e_{3}e_{3}^{*}e_{1}^{*}	\right\rbrace.$$  
     Since every non-zero element of $LI(F^{'})$ is uniquely expressed as one of the form $(a)$ and $ (b)$ (refer to Lemma \ref{lem forms}), all the non-zero elements of $LI(F^{'})$ are 
			$$w_{1},
			w_{2}, f_{1}^{*},
			w_{3}, f_{2}^{*}, f_{2}^{*}f_{1}^{*},
			w_{4}, f_{3}^{*},$$
			$$w_{5}, f_{4}^{*}, f_{4}^{*}f_{3}^{*},
			f_{1}, f_{1}f_{1}^{*}, 
			f_{2}, 
			f_{1}f_{2},
			f_{3}, f_{3}f_{3}^{*},
			f_{4},
			f_{3}f_{4}.$$
			 The two Leavitt inverse semigroups $LI(F)$ and $LI(F')$ are not graded isomorphic, as they do not have the same number of elements.
		\end{example}

		Let $K$ be a ring and $S$ a semigroup. We recall \cite{r10} that the corresponding semigroup ring as $KS$, and the resulting contracted semigroup ring as $K_{0}S$, where the zero element of $S$ is identified with the zero of $KS$. That is, $K_{0}S= KS/I$, where $I$ is the ideal of $KS$ generated by the zero element of $S$. We denote an arbitrary element of $K_{0}S$ by $\sum _{s\in S}a^{(s)}s$ (or $\sum _{s\in S\backslash \left\lbrace 0\right\rbrace }a^{(s)}s$), where $a^{(s)} \in K$, and all but finitely many of the $a^{(s)}$ are zero.

		\begin{lemma}{\label{lem1} \cite[Theorem 4.10]{r1}}
			Let $E$ be a directed graph and $K$ a  field, then the Leavitt path algebra $L_{K}(E) $ is isomorphic to the algebra $ K_{0}LI(E)/\left\langle v-\sum_{e\in s^{-1}(v)}ee^{*}\right\rangle $ where $K_{0}LI(E)$ is the contracted semigroup algebra and for each $v\in E^{0}$ has out-degree at least 2 .	
		\end{lemma}

 The contracted semigroup algebra $K_{0}LI(E)$ has the natural induced $\mathbb Z$-grading as $LI(E)$ is $\mathbb Z$-graded. And the Leavitt path algebra $L_{K}(E)$ is  $\mathbb Z$-graded via the length of paths in $E$. One has that $L_{K}(E) $ is $\mathbb Z$-graded 
 isomorphic to $ K_{0}LI(E)$.

		\begin{proposition}
			\label{thm111}
			Let $E$ and $F$ be connected graphs, $K$ a field and $\mathbb  Z$ an integer group. If $LI(E) \cong_{gr} LI(F)$ as $\mathbb  Z$-graded semigroups, then $L_{K}(E) \cong_{gr} L_{K}(F)$ as $\mathbb  Z$-graded Leavitt path algebras.
		\end{proposition}	
		
		\begin{proof}
			By Lemma $\ref{lem1}$, we see that $L_{K}(E)$ is isomorphic to the quotient of the contracted semigroup algebra $K_{0}LI(E)$ of $LI(E)$ by the ideal $I_{1}$ generated by elements $v-\sum _{e\in s^{-1}(v)}ee^{*}$ for $v\in E^{0}$ with the out-degree of $v$ at least 2. Similarly, $L_{K}(F)$ is isomorphic to the quotient of the contracted semigroup algebra $K_{0}LI(F)$ of $LI(F)$ by the ideal $I_{2}$ generated by elements $w-\sum _{f\in s^{-1}(w)}ff^{*}$ for $w\in F^{0}$ with the out-degree of $v$ at least 2. By \cite[Theorem 4.10]{r1}, we suppose $\phi$ is a $\mathbb  Z$-graded semigroup isomorphism from $LI(E)$ onto $LI(F)$, then  we have the following induced map
			\begin{align*}
				\widetilde{\phi}: K_{0}LI(E) &\longrightarrow K_{0}LI(F)\\
				pq^{*} &\longmapsto \phi (pq^{*}).
			\end{align*} Then we have that $\widetilde{\phi}$ is a $\mathbb  Z$-graded algebra isomorphism. By the proof of \cite[Theorem 4.10]{r1} we have $\widetilde{\phi}(I_1)=I_2$. One observes that $I_1$ and $I_2$ are both generated by homogeneous elements. Hence $$K_{0}LI(E)/I_{1} \cong_{gr} K_{0}LI(F)/I_{2}  $$            as $\mathbb  Z$-graded algebras.  Therefore we have    the $\mathbb  Z$-graded isomorphisms  $$L_K(E)\cong_{gr} K_{0}LI(E)/I_{1} \cong_{gr} K_{0}LI(F)/I_{2} \cong_{gr} L_K(F). $$     
		\end{proof} 

  \section{Acknowledgements} 
 This work is supported by the Natural Science Foundation of Anhui Province (No. 2108085QA05), the National Natural Science Foundations of China (Nos. 12101001, 12371015 and 12271442), the Excellent University Research and Innovation Team in Anhui Province (No. 2024AH010002),  and the Fundamental Research Funds for the Central Universities (SWU-XDJH202305).
		
		 {}

\begin{thebibliography}{}
            \bibitem{r4} G. Abrams, P. Ara, M. Siles Molina,  Leavitt path algebras,  Lecture Notes in Mathematics vol. 2191, Springer Verlag, 2017.    

            \bibitem{r5} P. Ara,  M. A. Moreno,  E. Pardo, {\it Nonstable K-theory for Graph Algebras}, Algebr. Represent. Theory
			$\mathbf{10}$(2):157--178, 2007.
    
            \bibitem{r6} G. Abrams, G. Aranda Pino, {\it The Leavitt path algebra of a graph}, J. Algebra $\mathbf{293}$(2):319-334, 2005.

            \bibitem{r16} A. Alahmadi, H. Alsulami, S.K. Jain, E. Zelmanov, {\it Leavitt path algebra of finite Gelfand-Kirillov dimension}, J. Algebra Appl. $\mathbf{11}$(6):1250225, 6 pp, 2012.

            \bibitem{r7} G. Abrams, G. Aranda Pino, M. Siles Molina, {\it Finite-dimensional Leavitt path algebras}, J.Pure Appl. Algebra $\mathbf{209}$:753-762, 2007.

            \bibitem{r12} H. Fan, Z. Wang, {\it On Leavitt inverse semigroups}, J. Algebra Appl. $\mathbf{20}$(9):2150158, 2021.
            
            \bibitem{r2} R. Hazrat,  Graded Rings and Graded Grothendieck Groups, Cambridge: Cambridge University Press, 2016.

            \bibitem{r9} R. Hazrat, {\it The graded structure of Leavitt path algebras}, Israel J.Math. $\mathbf{195}$:883-895, 2013.

            \bibitem{r10} R. Hazrat, Z. Mesyan, {\it Graded semigroups}, Israel J.Math. $\mathbf{253}$:249-319, 2023.


            

            \bibitem{r8} W.G. Leavitt, {\it The module type of a ring}, Trans. Amer. Math. Soc $\mathbf{103}$ (1):113-130, 1962.


          
            
			\bibitem{r1} J. Meakin, D. Milan and Z. Wang, {\it On a class of inverse semigroups related to Leavitt path algebras}, Adv. Math. $\mathbf{384}$: 100729, 2021.

   
		 \bibitem{NvObook}  C. N\u ast\u asescu, F. van Oystaeyen, Methods of graded rings, Lecture Notes in Mathematics 1836, Springer-Verlag, Berlin, 2004.
			
			\bibitem{r14} L. Va$\rm{\check{s}}$, {\it Realization of graded matrix algebras as Leavitt path algebras}, Beitr. Algebra Geom $\mathbf{61}$:771-781, 2020.

  
			
		\end{thebibliography}
	\end{document}